\documentclass[a4paper]{amsart}

\usepackage{amscd}
\usepackage{amsmath}
\usepackage{amssymb}
\usepackage{amsthm}
\usepackage{bbm}
\usepackage{stmaryrd}
\usepackage[arrow, matrix, curve]{xy}

\newif\ifpdf
\ifx\pdfoutput\undefined
   \pdffalse        
\else
   \pdfoutput=1     
   \pdftrue
\fi

\ifpdf
   \usepackage[pdftex]{graphicx}
\else
   \usepackage{graphicx}
\fi

\numberwithin{equation}{section} \swapnumbers

\newtheorem{satz}{Satz}[section]

\newtheorem{theorem}[satz]{Theorem}
\newtheorem{proposition}[satz]{Proposition}

\newtheorem{lemma}[satz]{Lemma}

\newtheorem{definition}[satz]{Definition}

\newtheorem{remark}[satz]{Remark}

\begin{document}

\hyphenation{do-mi-na-ted sa-tis-fy-ing cor-res-pon-ding Wel-fen-gar-ten}

\title[Integration with respect to an infinite dimensional L\'{e}vy process]{The It\^{o} integral with respect to an infinite dimensional L\'{e}vy process: A series approach}
\author{Stefan Tappe}
\address{Leibniz Universit\"{a}t Hannover, Institut f\"{u}r Mathematische Stochastik, Welfengarten 1, 30167 Hannover, Germany}
\email{tappe@stochastik.uni-hannover.de}
\thanks{The author is grateful to an anonymous referee for valuable comments and suggestions.}
\begin{abstract}
We present an alternative construction of the infinite dimensional It\^{o} integral with respect to a Hilbert space valued L\'{e}vy process. This approach is based on the well-known theory of real-valued stochastic integration, and the respective It\^{o} integral is given by a series of It\^{o} integrals with respect to standard L\'{e}vy processes. We also prove that this stochastic integral coincides with the It\^{o} integral that has been developed in the literature.
\end{abstract}
\maketitle

\section{Introduction}

The It\^{o} integral with respect to an infinite dimensional Wiener process has been developed in \cite{Da_Prato, Prevot-Roeckner, Atma-book}, and for the more general case of an infinite dimensional square-integrable martingale it has been defined in \cite{Metivier, P-Z-book}. In these references, one first constructs the It\^{o} integral for elementary processes, and then extends it via the It\^{o} isometry to a larger space, in which the space of elementary processes is dense.

For stochastic integrals with respect to a Wiener process, series expansions of the It\^{o} integral have been considered, e.g., in \cite{Hitsuda, Onno-Wiener, carteh}. Moreover, in the article \cite{van-Neerven-2005}, series expansions have been used in order to define the It\^{o} integral with respect to a Wiener process for deterministic integrands with values in a Banach space. Later, in \cite{van-Neerven} this theory has been extended to general integrands with values in UMD Banach spaces.

Best to the author's knowledge, a series approach for the construction of the It\^{o} integral with respect to an infinite dimensional L\'{e}vy process does not exist in the literature so far. The goal of the present paper is to provide such a construction, which is based on the real-valued It\^{o} integral, see, e.g., \cite{Applebaum, Jacod-Shiryaev, Protter}, and where the It\^{o} integral is given by a series of It\^{o} integrals with respect to real-valued L\'{e}vy processes. This approach has the advantage that we can use results from the finite dimensional case, and it might also be beneficial for lecturers teaching students who are already aware of the real-valued It\^{o} integral and have some background in Functional Analysis. In particular, it avoids the tedious procedure of proving that elementary processes are dense in the space of integrable processes.

In \cite{Onno-Levy}, the stochastic integral with respect to an infinite dimensional L\'{e}vy process is defined as a limit of Riemannian sums, and a series expansion is provided. A particular feature of \cite{Onno-Levy} is that stochastic integrals are considered as $L^2$-curves. The connection to the usual It\^{o} integral for a finite dimensional L\'{e}vy process has been established in \cite{Tappe}, see also Appendix B in \cite{Filipovic-Tappe}. Furthermore, we point out the articles \cite{Barbara-Integration} and \cite{Onno-Riedle}, where the theory of stochastic integration with respect to L\'{e}vy processes has been extended to Banach spaces.

The idea to use series expansions for the definition of the stochastic integral has also been utilized in the context of cylindrical processes, see \cite{Riedle-Wiener} for cylindrical Wiener processes and \cite{Applebaum-Riedle-Levy} for cylindrical L\'{e}vy processes.

The construction of the It\^{o} integral, which we present in this paper, is divided into the following steps:
\begin{itemize}
\item For a $H$-valued process $X$ (with $H$ denoting a separable Hilbert space) and a real-valued square-integrable martingale $M$ we define the It\^{o} integral
\begin{align*}
X \bullet M := \sum_{k \in \mathbb{N}} \big( \langle X, f_k \rangle_H \bullet M \big) f_k,
\end{align*}
where $(f_k)_{k \in \mathbb{N}}$ denotes an orthonormal basis of $H$, and $\langle X, f_k \rangle_H \bullet M$ denotes the real-valued It\^{o} integral. We will show that this definition does not depend on the choice of the orthonormal basis.

\item Based on the just defined integral, for a $\ell^2(H)$-valued process $X$ and a sequence $(M^j)_{j \in \mathbb{N}}$ of standard L\'{e}vy processes we define the It\^{o} integral as
\begin{align*}
\sum_{j \in \mathbb{N}} X^j \bullet M^j.
\end{align*}
For this, we will ensure convergence of the series.

\item In the next step, let $L$ denote a $\ell_{\lambda}^2$-valued L\'{e}vy process, where $\ell_{\lambda}^2$ is a weighted space of sequences (cf. \cite{fillnm}). From the L\'{e}vy process $L$ we can construct a sequence $(M^j)_{j \in \mathbb{N}}$ of standard L\'{e}vy processes, and for a $\ell^2(H)$-valued process $X$ we define the It\^{o} integral
\begin{align*}
X \bullet L := \sum_{j \in \mathbb{N}} X^j \bullet M^j.
\end{align*}

\item Finally, let $L$ be a general L\'{e}vy process on some separable Hilbert space $U$ with covariance operator $Q$. Then, there exist sequences of eigenvalues $(\lambda_j)_{j \in \mathbb{N}}$ and eigenvectors, which diagonalize the operator $Q$. Denoting by $L_2^0(H)$ an appropriate space of Hilbert Schmidt operators from $U$ to $H$, our idea is to utilize the integral from the previous step, and to define the It\^{o} integral for a $L_2^0(H)$-valued process $X$ as
\begin{align*}
X \bullet L := \Psi(X) \bullet \Phi(L),
\end{align*}
where $\Phi : U \rightarrow \ell_{\lambda}^2$ and $\Psi : L_2^0(H) \rightarrow \ell^2(H)$ are isometric isomorphisms such that $\Phi(L)$ is a $\ell_{\lambda}^2$-valued L\'{e}vy process. We will show that this definition does not depend on the choice of the eigenvalues and eigenvectors.
\end{itemize}
The remainder of this text is organized as follows: In Section~\ref{sec-prelim} we provide the required preliminaries and notation. After that, we start with the construction of the It\^{o} integral as outlined above. In Section~\ref{sec-step-1} we define the It\^{o} integral for $H$-valued processes with respect to a real-valued square-integrable martingale, and in Section~\ref{sec-step-2} we define the It\^{o} integral for $\ell^2(H)$-valued processes with respect to a sequence of standard L\'{e}vy processes. Section~\ref{sec-Wiener} gives a brief overview about L\'{e}vy processes in Hilbert spaces, together with the required results. 
Then, in Section~\ref{sec-step-3} we define the It\^{o} integral for $\ell^2(H)$-valued processes with respect to a $\ell_{\lambda}^2$-valued L\'{e}vy process, and in Section~\ref{sec-step-4} we define the It\^{o} integral in the general case, where the integrand is a $L_2^0(H)$-valued process and the integrator a general L\'{e}vy process on some separable Hilbert space $U$. We also prove the mentioned series representation of the stochastic integral, and show that it coincides with the usual It\^{o} integral, which has been developed in \cite{P-Z-book}.

\section{Preliminaries and notation}\label{sec-prelim}

In this section, we provide the required preliminary results and some basic notation. 
Throughout this text, let $(\Omega,\mathcal{F},(\mathcal{F}_t)_{t \geq 0},\mathbb{P})$ be a filtered probability space satisfying the usual conditions. For the upcoming results, let $E$ be a separable Banach space, and let $T > 0$ be a finite time horizon.

\begin{definition}
Let $p \geq 1$ be arbitrary.
\begin{enumerate}
\item We define the Lebesgue space
\begin{align*}
\mathcal{L}_T^p(E) := \mathcal{L}^p(\Omega,\mathcal{F}_T,\mathbb{P};\mathbb{D}([0,T];E)),
\end{align*}
where $\mathbb{D}([0,T];E)$ denotes the Skorokhod space consisting of all c\`{a}dl\`{a}g functions from $[0,T]$ to $E$, equipped with the supremum norm.

\item We denote by $\mathcal{A}_T^p(E)$ the space of all $E$-valued adapted processes $X \in \mathcal{L}_T^p(E)$.

\item We denote by $\mathcal{M}_T^p(E)$ the space of all $E$-valued martingales $M \in \mathcal{L}_T^p(E)$.

\item We define the factor spaces
\begin{align*}
M_T^p(E) := \mathcal{M}_T^p(E) / N, \quad A_T^p(E) := \mathcal{A}_T^p(E) / N, \quad L_T^p(E) := \mathcal{L}_T^p(E) / N,
\end{align*}
where $N \subset \mathcal{M}_T^p(E)$ denotes the subspace consisting of all $M \in \mathcal{M}_T^p(E)$ with $M = 0$ up to indistinguishability.
\end{enumerate}
\end{definition}

\begin{remark}
Let us emphasize the following:
\begin{enumerate}
\item Since the Skorokhod space $\mathbb{D}([0,T];E)$ equipped with the supremum norm is a Banach space, the Lebesgue space $L_T^p(E)$ equipped with the standard norm 
\begin{align*}
\| X \|_{L_T^p(E)} := \mathbb{E} \big[ \| X \|_E^p \big]^{1/p}
\end{align*}
is a Banach space, too.

\item By the completeness of the filtration $(\mathcal{F}_t)_{t \geq 0}$, adaptedness of an element $X \in L_T^p(E)$ does not depend on the choice of the representative. This ensures that the factor space $A_T^p(E)$ of adapted processes is well-defined.

\item The definition of $E$-valued martingales relies on the existence of conditional expectation in Banach spaces, which has been established in \cite[Prop.~1.10]{Da_Prato}.
\end{enumerate}
\end{remark}

Note that we have the inclusions
\begin{align*}
M_T^p(E) \subset A_T^p(E) \subset L_T^p(E).
\end{align*}
The following auxiliary result shows that these inclusions are closed.

\begin{lemma}\label{lemma-closed}
Let $p \geq 1$ be arbitrary. Then, the following statements are true:
\begin{enumerate}
\item $M_T^p(E)$ is closed in $A_T^p(E)$.

\item $A_T^p(E)$ is closed in $L_T^p(E)$.
\end{enumerate}
\end{lemma}

\begin{proof}
Let $(M^n)_{n \in \mathbb{N}} \subset M_T^p(E)$ be a sequence and let $M \in A_T^p(E)$ be such that $M^n \rightarrow M$ in $L_T^p(E)$. Furthermore, let $\tau \leq T$ be a bounded stopping time. Then we have
\begin{align*}
\mathbb{E} \big[ \| M_{\tau} \|_E^p \big] \leq \mathbb{E} \bigg[ \sup_{t \in [0,T]} \| M_t \|_E^p \bigg] < \infty,
\end{align*}
showing that $M_{\tau} \in L^p(\Omega,\mathcal{F}_{\tau},\mathbb{P};E)$.
Furthermore, we have
\begin{align*}
\mathbb{E} \big[ \| M_{\tau}^n - M_{\tau} \|_E^p \big] \leq \mathbb{E} \bigg[ \sup_{t \in [0,T]} \| M_t^n - M_t \|_E^p \bigg] \rightarrow 0.
\end{align*}
By Doob's optional stopping theorem (which also holds true for $E$-valued martingales, see \cite[Remark~2.2.5]{Prevot-Roeckner}), it follows that
\begin{align*}
\mathbb{E}[M_{\tau}] = \lim_{n \rightarrow \infty} \mathbb{E}[M_{\tau}^n] = \lim_{n \rightarrow \infty} \mathbb{E}[M_0^n] = \mathbb{E}[M_0]. 
\end{align*}
Using Doob's optional stopping theorem again, we conclude that $M \in M_T^p(E)$, proving the first statement.

Now, let $(X^n)_{n \in \mathbb{N}} \subset A_T^p(E)$ be a sequence and let $X \in L_T^p(E)$ be such that $X^n \rightarrow X$ in $L_T^p(E)$. Then, for each $t \in [0,T]$ we have
\begin{align*}
\mathbb{E} \big[ \| X_{t}^n - X_{t} \|_E^p \big] \leq \mathbb{E} \bigg[ \sup_{s \in [0,T]} \| X_s^n - X_s \|_E^p \bigg] \rightarrow 0,
\end{align*}
and hence $\mathbb{P}$--almost surely $X_t^{n_k} \rightarrow X_t$ for some subsequence $(n_k)_{k \in \mathbb{N}}$,
showing that $X_t$ is $\mathcal{F}_t$-measurable. This proves $X \in A_T^p(E)$, providing the second statement. 
\end{proof}

Note that, by Doob's martingale inequality \cite[Thm.~2.2.7]{Prevot-Roeckner}, for $p > 1$ an equivalent norm on $M_T^p(E)$ is given by
\begin{align*}
\| M \|_{M_T^p(E)} := \mathbb{E} \big[ \| M_T \|_E^p \big]^{1/p}.
\end{align*}
Furthermore, if $E=H$ is a separable Hilbert space, then $M_T^2(H)$ is a separable Hilbert space equipped with the inner product
\begin{align*}
\langle M,N \rangle_{M_T^2(H)} := \mathbb{E} \big[ \langle M_T,N_T \rangle_H \big].
\end{align*}
Finally, we recall the following result about series of pairwise orthogonal vectors in Hilbert spaces.

\begin{lemma}\label{lemma-series-Hilbert}
Let $H$ be a separable Hilbert space and let $(h_n)_{n \in \mathbb{N}} \subset H$ be a sequence with 
$\langle h_n,h_m \rangle_H = 0$ for $n \neq m$. 
Then, the following statements are equivalent:
\begin{enumerate}
\item The series $\sum_{n=1}^{\infty} h_n$ converges in $H$.

\item The series $\sum_{n \in \mathbb{N}} h_n$ converges unconditionally in $H$.

\item We have $\sum_{n=1}^{\infty} \| h_n \|_H^2 < \infty$.
\end{enumerate}
If the previous conditions are satisfied, then we have
\begin{align*}
\bigg\| \sum_{n=1}^{\infty} h_n \bigg\|_H^2 = \sum_{n=1}^{\infty} \| h_n \|_H^2.
\end{align*}
\end{lemma}

\begin{proof}
This follows from \cite[Thm.~12.6]{Rudin} and \cite[Satz~V.4.8]{Werner}.
\end{proof}

\section{The It\^{o} integral with respect to a real-valued square-integrable martingale}\label{sec-step-1}

In this section, we define the It\^{o} integral for Hilbert space valued processes with respect to a real-valued, square-integrable martingale, which is based on the real-valued It\^{o} integral.

In what follows, let $H$ be s separable Hilbert space, and let $T > 0$ be a finite time horizon. Furthermore, let $M \in \mathcal{M}_T^2(\mathbb{R})$ be a square-integrable martingale. Recall that the quadratic variation $\langle M,M \rangle$ is the (up to indistinguishability) unique real-valued, non-decreasing, predictable process with $\langle M,M \rangle_0 = 0$ such that $M^2 - \langle M,M \rangle$ is a martingale.

\begin{proposition}\label{prop-int-well-def}
Let $X$ be a $H$-valued, predictable process with
\begin{align}\label{int-fk-exists}
\mathbb{E} \bigg[ \int_0^T \| X_s \|_H^2 d \langle M,M \rangle_s \bigg] < \infty.
\end{align}
Then, for every orthonormal basis $(f_k)_{k \in \mathbb{N}}$ of $H$ the series
\begin{align}\label{series-int-fk}
\sum_{k \in \mathbb{N}} \big( \langle X, f_k \rangle_H \bullet M \big) f_k
\end{align}
converges unconditionally in $M_T^2(H)$, and its value does not depend on the choice of the orthonormal basis $(f_k)_{k \in \mathbb{N}}$.
\end{proposition}

\begin{proof}
Let $(f_k)_{k \in \mathbb{N}}$ be an orthonormal basis of $H$.
For $j,k \in \mathbb{N}$ with $j \neq k$ we have
\begin{equation}\label{orthogonal-in-MT2}
\begin{aligned}
&\big\langle \big( \langle X, f_j \rangle_H \bullet M \big) f_j, \big( \langle X, f_k \rangle_H \bullet M \big) f_k \big\rangle_{M_T^2(H)} 
\\ &= \mathbb{E} \bigg[ \Big\langle \bigg( \int_0^T \langle X_s, f_j \rangle_H d M_s \bigg) f_j, \bigg( \int_0^T \langle X_s, f_k \rangle_H d M_s \bigg) f_k \Big\rangle_H \bigg]
\\ &= \mathbb{E} \bigg[ \bigg( \int_0^T \langle X_s, f_j \rangle_H d M_s \bigg) \bigg( \int_0^T \langle X_s, f_k \rangle_H d M_s \bigg) \langle f_j,f_k \rangle_H \bigg] = 0.
\end{aligned}
\end{equation}
Moreover, by the It\^{o} isometry for the real-valued It\^{o} integral and the monotone convergence theorem we obtain
\begin{equation}\label{calc-Ito-isom}
\begin{aligned}
&\sum_{k=1}^{\infty} \big\| \big( \langle X, f_k \rangle_H \bullet M \big) f_k \big\|_{M_T^2(H)}^2 = \sum_{k=1}^{\infty} \mathbb{E} \Bigg[ \bigg\| \bigg( \int_0^T \langle X_s, f_k \rangle_H d M_s \bigg) f_k \bigg\|_H^2 \Bigg]
\\ &= \sum_{k=1}^{\infty} \mathbb{E} \Bigg[ \bigg| \int_0^T \langle X_s, f_k \rangle_H d M_s \bigg|^2 \Bigg] = \sum_{k=1}^{\infty} \mathbb{E} \bigg[ \int_0^T | \langle X_s, f_k \rangle_H |^2 d\langle M,M \rangle_s \bigg] 
\\ &= \mathbb{E} \bigg[ \int_0^T \sum_{k=1}^{\infty} | \langle X_s, f_k \rangle_H |^2 d\langle M,M \rangle_s \bigg] = \mathbb{E} \bigg[ \int_0^T \| X_s \|_H^2 d\langle M,M \rangle_s \bigg].
\end{aligned}
\end{equation}
Therefore, by (\ref{int-fk-exists}) and Lemma~\ref{lemma-series-Hilbert}, the series (\ref{series-int-fk}) converges unconditionally in $M_T^2(H)$.

Now, let $(g_k)_{k \in \mathbb{N}}$ be another orthonormal basis of $H$. We define $\mathbb{J}^f,\mathbb{J}^g \in M_T^2(H)$ by
\begin{align*}
\mathbb{J}^f := \sum_{k=1}^{\infty} \big( \langle X,f_k \rangle_H \bullet M \big) f_k \quad \text{and} \quad \mathbb{J}^g := \sum_{k=1}^{\infty} \big( \langle X,g_k \rangle_H \bullet M \big) g_k.
\end{align*}
Let $h \in H$ be arbitrary. Then we have
\begin{align*}
\langle h,\mathbb{J}^f \rangle_H, \langle h,\mathbb{J}^g \rangle_H \in M_T^2(\mathbb{R})
\end{align*}
and the identity
\begin{align*}
&\| \langle h, \mathbb{J}^f \rangle_H - \langle h,X \rangle_H \bullet M \|_{M_T^2(\mathbb{R})}^2 
\\ &= \bigg\| \Big\langle h, \sum_{k=1}^{\infty} \big( \langle X,f_k \rangle_H \bullet M \big) f_k \Big\rangle_H - \langle h,X \rangle_H \bullet M \bigg\|_{M_T^2(\mathbb{R})}^2
\\ &= \bigg\| \sum_{k=1}^{\infty} \big( \langle h,f_k \rangle_H \langle f_k,X \rangle_H \bullet M \big) - \langle h,X \rangle_H \bullet M \bigg\|_{M_T^2(\mathbb{R})}^2
\\ &= \lim_{n \rightarrow \infty} \bigg\| \sum_{k=1}^{n} \big( \langle h,f_k \rangle_H \langle f_k,X \rangle_H \bullet M \big) - \langle h,X \rangle_H \bullet M \bigg\|_{M_T^2(\mathbb{R})}^2
\\ &= \lim_{n \rightarrow \infty} \bigg\| \bigg( \sum_{k=1}^{n} \langle h,f_k \rangle_H \langle f_k,X \rangle_H - \langle h,X \rangle_H \bigg) \bullet M \bigg\|_{M_T^2(\mathbb{R})}^2.
\end{align*}
For all $x \in H$ we have
\begin{align*}
\bigg| \sum_{k=1}^{n} \langle x, f_k \rangle_H \langle f_k,h \rangle_H - \langle x,h \rangle_H \bigg|^2 \rightarrow 0 \quad \text{as $n \rightarrow \infty$,}
\end{align*}
and, by the Cauchy-Schwarz inequality,
\begin{align*}
&\bigg| \sum_{k=1}^{n} \langle x, f_k \rangle_H \langle f_k,h \rangle_H - \langle x,h \rangle_H \bigg|^2 = \bigg| \sum_{k=n+1}^{\infty} \langle x, f_k \rangle_H \langle f_k,h \rangle_H \bigg|^2
\\ &\leq \bigg( \sum_{k=1}^\infty | \langle x,f_k \rangle_H |^2 \bigg) \bigg( \sum_{k=1}^\infty | \langle f_k,h \rangle_H |^2 \bigg) = \| x \|_H^2 \| h \|_H^2 \quad \text{for each $n \in \mathbb{N}$.}
\end{align*}
Therefore, by the It\^{o} isometry for the real-valued It\^{o} integral and Lebesgue's dominated convergence theorem together with (\ref{int-fk-exists}) we obtain
\begin{align*}
&\| \langle h, \mathbb{J}^f \rangle_H - \langle h,X \rangle_H \bullet M \|_{M_T^2(\mathbb{R})}^2
\\ &= \lim_{n \rightarrow \infty} \mathbb{E} \Bigg[ \bigg| \int_0^T \bigg( \sum_{k=1}^{n} \langle h,f_k \rangle_H \langle f_k,X_s \rangle_H - \langle h,X_s \rangle_H \bigg) d M_s \bigg|^2 \Bigg]
\\ &= \lim_{n \rightarrow \infty} \mathbb{E} \bigg[ \int_0^T \bigg| \sum_{k=1}^{n} \langle h,f_k \rangle_H \langle f_k,X \rangle_H - \langle h,X \rangle_H \bigg|^2 d\langle M,M \rangle_s \bigg] = 0.
\end{align*}
Analogously, we prove that
\begin{align*}
\| \langle h, \mathbb{J}^g \rangle_H - \langle h,X \rangle_H \bullet M \|_{M_T^2(\mathbb{R})}^2 = 0.
\end{align*}
Therefore, denoting by $\tilde{\mathbb{J}}^f,\tilde{\mathbb{J}}^g \in \mathcal{M}_T^2(H)$ representatives of $\mathbb{J}^f$, $\mathbb{J}^g$, we obtain
\begin{align*}
\langle h,\tilde{\mathbb{J}}_T^f \rangle_H = \langle h,\tilde{\mathbb{J}}_T^g \rangle_H \quad \text{for all $h \in H$,} \quad \text{$\mathbb{P}$--almost surely.}
\end{align*}
By separability of $H$, we deduce that
\begin{align*}
\langle h,\tilde{\mathbb{J}}_T^f \rangle_H = \langle h,\tilde{\mathbb{J}}_T^g \rangle_H \quad \text{$\mathbb{P}$--almost surely,} \quad \text{for all $h \in H$.}
\end{align*}
Consequently, we have 
\begin{align*}
\tilde{\mathbb{J}}_T^f = \tilde{\mathbb{J}}_T^g \quad \text{$\mathbb{P}$--almost surely,} 
\end{align*}
implying $\mathbb{J}^f = \mathbb{J}^g$. This proves that the value of the series (\ref{series-int-fk}) does not depend on the choice of the orthonormal basis.
\end{proof}

Now, Proposition~\ref{prop-int-well-def} gives rise to the following definition:

\begin{definition}
For every $H$-valued, predictable process $X$ satisfying (\ref{int-fk-exists}) we define the \emph{It\^{o} integral} $X \bullet M = (\int_0^t X_s dM_s)_{t \in [0,T]}$ as
\begin{align}\label{def-Ito-H}
X \bullet M := \sum_{k \in \mathbb{N}} \big( \langle X, f_k \rangle_H \bullet M \big) f_k,
\end{align}
where $(f_k)_{k \in \mathbb{N}}$ denotes an orthonormal basis of $H$.
\end{definition}

According to Proposition~\ref{prop-int-well-def}, the Definition (\ref{def-Ito-H}) of the It\^{o} integral is independent of the choice of the orthonormal basis $(f_k)_{k \in \mathbb{N}}$, and the integral process $X \bullet M$ belongs to $M_T^2(H)$.

\begin{remark}
As the proof of Proposition~\ref{prop-int-well-def} shows, the components of the It\^{o} integral $X \bullet M$ are pairwise orthogonal elements of the Hilbert space $M_T^2(H)$.
\end{remark}

\begin{proposition}\label{prop-Ito-isom-1}
For every $H$-valued, predictable process $X$ satisfying (\ref{int-fk-exists}) we have the \emph{It\^{o} isometry}
\begin{align*}
\mathbb{E} \Bigg[ \bigg\| \int_0^T X_s dM_s \bigg\|_H^2 \Bigg] = \mathbb{E} \bigg[ \int_0^T \| X_s \|_H^2 d \langle M,M \rangle_s \bigg]. 
\end{align*}
\end{proposition}

\begin{proof}
Let $(f_k)_{k \in \mathbb{N}}$ be an orthonormal basis of $H$. According to (\ref{orthogonal-in-MT2}) we have
\begin{align*}
\big\langle \big( \langle X, f_j \rangle_H \bullet M \big) f_j, \big( \langle X, f_k \rangle_H \bullet M \big) f_k \big\rangle_{M_T^2(H)} = 0 \quad \text{for $j \neq k$.}
\end{align*}
Thus, by Lemma~\ref{lemma-series-Hilbert} and (\ref{calc-Ito-isom}) we obtain
\begin{align*}
&\mathbb{E} \Bigg[ \bigg\| \int_0^T X_s dM_s \bigg\|_H^2 \Bigg] = \| X \bullet M \|_{M_T^2(H)}^2 = \bigg\| \sum_{k = 1}^{\infty} \big( \langle X, f_k \rangle_H \bullet M \big) f_k \bigg\|_{M_T^2(H)}^2
\\ &= \sum_{k=1}^{\infty} \big\| \big( \langle X, f_k \rangle_H \bullet M \big) f_k \big\|_{M_T^2(H)}^2 = \mathbb{E} \bigg[ \int_0^T \| X_s \|_H^2 d \langle M,M \rangle_s \bigg],
\end{align*}
finishing the proof.
\end{proof}

\begin{proposition}\label{prop-simple-1}
Let $X$ be a $H$-valued simple process of the form
\begin{align*}
X = X_0 \mathbbm{1}_{\{0\}} + \sum_{i=1}^n X_i \mathbbm{1}_{(t_i,t_{i+1}]}
\end{align*}
with $0 = t_1 < \ldots < t_{n+1} = T$ and $\mathcal{F}_{t_i}$-measurable random variables $X_i : \Omega \rightarrow H$ for $i = 0,\ldots,n$. Then, we have
\begin{align*}
X \bullet M = \sum_{i=1}^n X_i (M^{t_{i+1}} - M^{t_i}).
\end{align*}
\end{proposition}

\begin{proof}
Let $(f_k)_{k \in \mathbb{N}}$ be an orthonormal basis of $H$. Then, for each $k \in \mathbb{N}$ the process $\langle X,f_k \rangle$ is a real-valued simple process with representation
\begin{align*}
\langle X,f_k \rangle_H = \langle X_0,f_k \rangle_H \mathbbm{1}_{\{0\}} + \sum_{i=1}^n \langle X_i,f_k \rangle_H \mathbbm{1}_{(t_i,t_{i+1}]}
\end{align*}
Thus, by the definition of the real-valued It\^{o} integral for simple processes we obtain
\begin{align*}
X \bullet M &= \sum_{k \in \mathbb{N}} \big( \langle X, f_k \rangle_H \bullet M \big) f_k = \sum_{k \in \mathbb{N}} \bigg( \sum_{i=1}^n \langle X_i, f_k \rangle_H (M^{t_{i+1}} - M^{t_i}) \bigg) f_k
\\ &= \sum_{i=1}^n \bigg( \sum_{k \in \mathbb{N}} \langle X_i, f_k \rangle_H f_k \bigg) (M^{t_{i+1}} - M^{t_i}) = \sum_{i=1}^n X_i (M^{t_{i+1}} - M^{t_i}),
\end{align*}
finishing the proof.
\end{proof}

\begin{lemma}\label{lemma-I-conv}
Let $X$ be a $H$-valued, predictable process satisfying (\ref{int-fk-exists}). Then, for every orthonormal basis $(f_k)_{k \in \mathbb{N}}$ of $H$ we have
\begin{align*}
\sum_{k=1}^{\infty} | \langle X,f_k \rangle_H |^2 \bullet \langle M, M \rangle = \| X \|_H^2 \bullet \langle M, M \rangle,
\end{align*}
where the convergence takes place in $A_T^1(\mathbb{R})$.
\end{lemma}

\begin{proof}
We define the integral process
\begin{align}\label{def-I}
\mathbb{I} := \| X \|_H^2 \bullet \langle M, M \rangle
\end{align}
and the sequence $(\mathbb{I}^n)_{n \in \mathbb{N}}$ of partial sums by
\begin{align}\label{def-I-n}
\mathbb{I}^n := \sum_{k=1}^n | \langle X,f_k \rangle_H |^2 \bullet \langle M, M \rangle.
\end{align}
By (\ref{int-fk-exists}) we have $\mathbb{I} \in A_T^1(\mathbb{R})$ and $(\mathbb{I}^n)_{n \in \mathbb{N}} \subset A_T^1(\mathbb{R})$.
Furthermore, by Lebesgue's dominated convergence theorem we have
\begin{align*}
\| \mathbb{I} - \mathbb{I}^n \|_{L_T^1(\mathbb{R})} &= \mathbb{E} \bigg[ \sup_{t \in [0,T]} | \mathbb{I}_t - \mathbb{I}_t^n | \bigg] = \mathbb{E} \Bigg[ \sup_{t \in [0,T]} \bigg| \int_0^t \sum_{k=n+1}^{\infty} |\langle X_s,f_k \rangle|^2 d \langle M, M \rangle_s \bigg| \Bigg]
\\ &= \mathbb{E} \bigg[ \int_0^T \sum_{k=n+1}^{\infty} |\langle X_s,f_k \rangle|^2 d \langle M, M \rangle_s \bigg] \rightarrow 0 \quad \text{for $n \rightarrow \infty$,}
\end{align*}
which concludes the proof.
\end{proof}

\begin{remark}\label{remark-cov-prod}
As a consequence of the Doob-Meyer decomposition theorem, for two square-integrable martingales $X,Y \in \mathcal{M}_T^2(H)$ there exists a (up to indistinguishability) unique real-valued, predictable process $\langle X,Y \rangle$ with finite variation paths and $\langle X,Y \rangle_0 = 0$ such that $\langle X,Y \rangle_H - \langle X,Y \rangle$ is a martingale.
\end{remark}

\begin{proposition}\label{prop-variation}
For every $H$-valued, predictable process $X$ satisfying (\ref{int-fk-exists}) we have
\begin{align*}
\langle X \bullet M, X \bullet M \rangle = \| X \|_H^2 \bullet \langle M,M \rangle.
\end{align*}
\end{proposition}

\begin{proof}
Let $(f_k)_{k \in \mathbb{N}}$ be an orthonormal basis of $H$. We define the process $\mathbb{J} := X \bullet M$
and the sequence $(\mathbb{J}^n)_{n \in \mathbb{N}}$ of partial sums by
\begin{align*}
\mathbb{J}^n := \sum_{k=1}^n \big( \langle X,f_k \rangle_H \bullet M \big) f_k.
\end{align*}
By Proposition~\ref{prop-int-well-def} we have 
\begin{align}\label{J-conv}
\mathbb{J}^n \rightarrow \mathbb{J} \quad \text{in $M_T^2(H)$.}
\end{align}
Defining the integral process $\mathbb{I}$ by (\ref{def-I})
and the sequence $(\mathbb{I}^n)_{n \in \mathbb{N}}$ of partial sums by (\ref{def-I-n}), using Lemma~\ref{lemma-I-conv} we have 
\begin{align}\label{I-conv}
\mathbb{I}^n \rightarrow \mathbb{I} \quad \text{in $A_T^1(\mathbb{R})$.} 
\end{align}
Furthermore, we define the process $M \in A_T^1(\mathbb{R})$ and the sequence $(M^n)_{n \in \mathbb{N}} \subset A_T^1(\mathbb{R})$ as
\begin{align*}
M &:= \| \mathbb{J} \|_H^2 - \mathbb{I},
\\ M^n &:= \| \mathbb{J}^n \|_H^2 - \mathbb{I}^n, \quad n \in \mathbb{N}.
\end{align*}
Then we have $(M^n)_{n \in \mathbb{N}} \subset M_T^1(\mathbb{R})$. Indeed, for each $n \in \mathbb{N}$ we have
\begin{align*}
M^n &= \bigg\| \sum_{k=1}^n \big( \langle X,f_k \rangle_H \bullet M \big) f_k \bigg\|_H^2 - \sum_{k=1}^n | \langle X,f_k \rangle_H |^2 \bullet \langle M,M \rangle
\\ &= \sum_{k=1}^n \big\| \big( \langle X,f_k \rangle_H \bullet M \big) f_k \big\|_H^2 - \sum_{k=1}^n | \langle X,f_k \rangle_H |^2 \bullet \langle M,M \rangle
\\ &= \sum_{k=1}^n \big( | \langle X,f_k \rangle_H \bullet M |^2 - | \langle X,f_k \rangle_H |^2 \bullet \langle M,M \rangle \big).
\end{align*}
For every $k \in \mathbb{N}$ the quadratic variation of the real-valued process $\langle X,f_k \rangle_H \bullet M$ is given by
\begin{align*}
\langle \langle X,f_k \rangle_H \bullet M, \langle X,f_k \rangle_H \bullet M \rangle = | \langle X,f_k \rangle_H |^2 \bullet \langle M,M \rangle,
\end{align*}
see, e.g. \cite[Thm.~I.4.40.d]{Jacod-Shiryaev}, which shows that $M^n$ is a martingale. Since $M^n \in A_T^1(\mathbb{R})$, we deduce that $M^n \in M_T^1(\mathbb{R})$.

Next, we prove that $M^n \rightarrow M$ in $A_T^1(\mathbb{R})$. Indeed, since
\begin{align*}
| \, \| \mathbb{J} \|_H^2 - \| \mathbb{J}^n \|_H^2 \, | \leq \| \mathbb{J} - \mathbb{J}^n \|_H^2 + 2 \| \mathbb{J} \|_H \| \mathbb{J} - \mathbb{J}^n \|_H,
\end{align*}
by the Cauchy-Schwarz inequality and (\ref{J-conv}) we obtain
\begin{align*}
&\| \,  \| \mathbb{J} \|_H^2 - \| \mathbb{J}^n \|_H^2 \, \|_{L_T^1(\mathbb{R})} = \mathbb{E} \bigg[ \sup_{t \in [0,T]} | \, \| \mathbb{J}_t \|_H^2 - \| \mathbb{J}_t^n \|_H^2 \, | \bigg]
\\ &\leq \mathbb{E} \bigg[ \sup_{t \in [0,T]} \| \mathbb{J}_t - \mathbb{J}_t^n \|_H^2 \bigg] + 2 \mathbb{E} \bigg[ \sup_{t \in [0,T]} \| \mathbb{J}_t \|_H \| \mathbb{J}_t - \mathbb{J}_t^n \|_H \bigg]
\\ &\leq \mathbb{E} \bigg[ \sup_{t \in [0,T]} \| \mathbb{J}_t - \mathbb{J}_t^n \|_H^2 \bigg] + 2 \mathbb{E} \bigg[ \sup_{t \in [0,T]} \| \mathbb{J}_t \|_H^2 \bigg]^{1/2} \mathbb{E} \bigg[ \sup_{t \in [0,T]} \| \mathbb{J}_t - \mathbb{J}_t^n \|_H^2 \bigg]^{1/2}
\\ &= \| \mathbb{J} - \mathbb{J}^n \|_{L_T^2(H)}^2 + 2 \| \mathbb{J} \|_{L_T^2(H)} \| \mathbb{J} - \mathbb{J}^n \|_{L_T^2(H)} \rightarrow 0.
\end{align*}
Therefore, together with (\ref{I-conv}) we get
\begin{align*}
\| M - M^n \|_{L_T^1(\mathbb{R})} \leq \| \,  \| \mathbb{J} \|^2 - \| \mathbb{J}^n \|^2 \, \|_{L_T^1(\mathbb{R})} + \| \mathbb{I} - \mathbb{I}^n \|_{L_T^1(\mathbb{R})} \rightarrow 0,
\end{align*}
showing that $M^n \rightarrow M$ in $A_T^1(\mathbb{R})$. Now, Lemma~\ref{lemma-closed} yields that $M \in M_T^1(\mathbb{R})$, which concludes the proof.
\end{proof}

\begin{theorem}\label{thm-cov}
Let $N \in \mathcal{M}_T^2(\mathbb{R})$ be another square-integrable martingale, and let $X,Y$ be two $H$-valued, predictable processes satisfying (\ref{int-fk-exists}) and
\begin{align}\label{Y-int}
\mathbb{E} \bigg[ \int_0^T \| Y_s \|_H^2 d \langle N,N \rangle_s \bigg] < \infty.
\end{align} 
Then we have
\begin{align}\label{cov-id}
\langle X \bullet M, Y \bullet N \rangle = \langle X, Y \rangle_H \bullet \langle M, N \rangle.
\end{align}
\end{theorem}

\begin{proof}
Using Proposition~\ref{prop-variation} and the identities
\begin{align*}
\langle x,y \rangle_H &= \frac{1}{4} \big( \| x+y \|_H^2 - \| x-y \|_H^2 \big), \quad x,y \in H,
\\ \langle M,N \rangle &= \frac{1}{4} \big( \langle M+N,M+N \rangle - \langle M-N, M-N \rangle \big),
\end{align*}
identity (\ref{cov-id}) follows from a straightforward calculation.
\end{proof}

\begin{proposition}\label{prop-cor-integrals}
Let $N \in \mathcal{M}_T^2(\mathbb{R})$ be another square-integrable martingale such that $\langle M,N \rangle = 0$, and let $X,Y$ be two $H$-valued, predictable processes satisfying (\ref{int-fk-exists}) and (\ref{Y-int}).
Then we have
\begin{align*}
\langle X \bullet M, Y \bullet N \rangle_{M_T^2(H)} = 0.
\end{align*}
\end{proposition}

\begin{proof}
Using Remark~\ref{remark-cov-prod}, Theorem~\ref{thm-cov} and the hypothesis $\langle M,N \rangle = 0$ we obtain
\begin{align*}
&\langle X \bullet M, Y \bullet N \rangle_{M_T^2(H)} = \mathbb{E} \bigg[ \Big\langle \int_0^T X_s dM_s, \int_0^T Y_s dN_s \Big\rangle_H \bigg] 
\\ &= \mathbb{E} \bigg[ \Big\langle \int_0^T X_s dM_s, \int_0^T Y_s dN_s \Big\rangle \bigg] = \mathbb{E} \bigg[ \int_0^T \langle X_s, Y_s \rangle_H d \langle M, N \rangle_s \bigg] = 0,
\end{align*}
completing the proof.
\end{proof}

\section{The It\^{o} integral with respect to a sequence of standard L\'{e}vy processes}\label{sec-step-2}

In this section, we introduce the It\^{o} integral for $\ell^2(H)$-valued processes with respect to a sequence of standard L\'{e}vy processes, which is based on the It\^{o} integral (\ref{def-Ito-H}) from the previous section.
We define the space of sequences
\begin{align*}
\ell^2(H) := \bigg\{ (h^j)_{j \in \mathbb{N}} \subset H : \sum_{j=1}^{\infty} \| h^j \|_H^2 < \infty \bigg\},
\end{align*}
which, equipped with the inner product
\begin{align*}
\langle h,g \rangle_{\ell^2(H)} = \sum_{j=1}^{\infty} \langle h^j,g^j \rangle_H
\end{align*}
is a separable Hilbert space. 

\begin{definition}
A sequence $(M^j)_{j \in \mathbb{N}}$ of real-valued L\'{e}vy processes is called a \emph{sequence of standard L\'{e}vy processes} if it consists of square-integrable martingales with $\langle M^j,M^k \rangle_t = \delta_{jk} \cdot t$ for all $j,k \in \mathbb{N}$. Here $\delta_{jk}$ denotes the Kronecker delta
\begin{align*}
\delta_{jk} =
\begin{cases}
1, & \text{if } j = k,
\\ 0, & \text{if } j \neq k.
\end{cases}
\end{align*}
\end{definition}

For the rest of this section, let $(M^j)_{j \in \mathbb{N}}$ be a sequence of standard L\'{e}vy processes.

\begin{proposition}\label{prop-int-conv-2}
For every $\ell^2(H)$-valued, predictable process $X$ with
\begin{align}\label{int-cond-2}
\mathbb{E} \bigg[ \int_0^T \| X_s \|_{\ell^2(H)}^2 ds \bigg] < \infty
\end{align}
the series 
\begin{align}\label{series-int-2}
\sum_{j \in \mathbb{N}} X^j \bullet M^j
\end{align}
converges unconditionally in $M_T^2(H)$.
\end{proposition}

\begin{proof}
For $j,k \in \mathbb{N}$ with $j \neq k$ we have $\langle M^j,M^k \rangle = 0$, and hence, by Proposition~\ref{prop-cor-integrals} we obtain
\begin{align}\label{orthogonal-integrals}
\langle X^j \bullet M^j, X^k \bullet M^k \rangle_{M_T^2(H)} = 0.
\end{align}
Moreover, by the It\^{o} isometry (Proposition~\ref{prop-Ito-isom-1}) and the monotone convergence theorem we have
\begin{equation}\label{calc-Ito-isom-2}
\begin{aligned}
&\sum_{j=1}^{\infty} \| X^j \bullet M^j \|_{M_T^2(H)}^2 = \sum_{j=1}^{\infty} \mathbb{E} \Bigg[ \bigg\| \int_0^T X_s^j dM_s^j \bigg\|_H^2 \Bigg] = \sum_{j=1}^{\infty} \mathbb{E} \bigg[ \int_0^T \| X_s^j \|_H^2 ds \bigg] 
\\ &= \mathbb{E} \bigg[ \int_0^T \sum_{j=1}^{\infty} \| X_s^j \|_H^2 ds \bigg] = \mathbb{E} \bigg[ \int_0^T \| X_s \|_{\ell^2(H)}^2 ds \bigg].
\end{aligned}
\end{equation}
Thus, by (\ref{int-cond-2}) and Lemma~\ref{lemma-series-Hilbert}, the series (\ref{series-int-2}) converges unconditionally in $M_T^2(H)$.
\end{proof}

Therefore, for a $\ell^2(H)$-valued, predictable process $X$ satisfying (\ref{int-cond-2}) we can define the It\^{o} integral as the series (\ref{series-int-2}).

\begin{remark}\label{rem-orth}
As the proof of Proposition~\ref{prop-int-conv-2} shows, the components of the It\^{o} integral $\sum_{j \in \mathbb{N}} X^j \bullet M^j$ are pairwise orthogonal elements of the Hilbert space $M_T^2(H)$.
\end{remark}

\begin{proposition}\label{prop-Ito-isom-step-2}
For each $\ell^2(H)$-valued, predictable process $X$ satisfying (\ref{int-cond-2}) we have the \emph{It\^{o} isometry}
\begin{align*}
\mathbb{E} \Bigg[ \bigg\| \sum_{j \in \mathbb{N}} \int_0^T X_s^j dM_s^j \bigg\|_H^2 \Bigg] = \mathbb{E} \bigg[ \int_0^T \| X_s \|_{\ell^2(H)}^2 ds \bigg].
\end{align*}
\end{proposition}

\begin{proof}
Using (\ref{orthogonal-integrals}), Lemma~\ref{lemma-series-Hilbert} and identity (\ref{calc-Ito-isom-2}) we obtain
\begin{align*}
&\mathbb{E} \Bigg[ \bigg\| \sum_{j=1}^{\infty} \int_0^T X_s^j dM_s^j \bigg\|_H^2 \Bigg] = \bigg\| \sum_{j=1}^{\infty} X^j \bullet M^j \bigg\|_{M_T^2(H)}^2 
\\ &= \sum_{j=1}^{\infty} \| X^j \bullet M^j \|_{M_T^2(H)}^2 = \mathbb{E} \bigg[ \int_0^T \| X_s \|_{\ell^2(H)}^2 ds \bigg],
\end{align*}
completing the proof.
\end{proof}

\begin{proposition}\label{prop-simple-2}
Let $X$ be a $\ell^2(H)$-valued simple process of the form
\begin{align*}
X = X_0 \mathbbm{1}_{\{0\}} + \sum_{i=1}^n X_i \mathbbm{1}_{(t_i,t_{i+1}]}
\end{align*}
with $0 = t_1 < \ldots < t_{n+1} = T$ and $\mathcal{F}_{t_i}$-measurable random variables $X_i : \Omega \rightarrow \ell^2(H)$ for $i = 0,\ldots,n$. Then we have
\begin{align*}
X \bullet M = \sum_{i=1}^n \sum_{j \in \mathbb{N}} X_i^j \big( (M^j)^{t_{i+1}} - (M^j)^{t_i} \big).
\end{align*}
\end{proposition}

\begin{proof}
For each $j \in \mathbb{N}$ the process $X^j$ is a $H$-valued simple process having the representation
\begin{align*}
X^j = X_0^j \mathbbm{1}_{\{0\}} + \sum_{i=1}^n X_i^j \mathbbm{1}_{(t_i,t_{i+1}]}.
\end{align*}
Hence, by Proposition~\ref{prop-simple-1} we obtain
\begin{align*}
X \bullet M &= \sum_{j \in \mathbb{N}} X^j \bullet M^j = \sum_{j \in \mathbb{N}} \sum_{i=1}^n X_i^j \big( (M^j)^{t_{i+1}} - (M^j)^{t_i} \big)
\\ &= \sum_{i=1}^n \sum_{j \in \mathbb{N}} X_i^j \big( (M^j)^{t_{i+1}} - (M^j)^{t_i} \big),
\end{align*}
which finishes the proof.
\end{proof}

\section{L\'{e}vy processes in Hilbert spaces}\label{sec-Wiener}

In this section, we provide the required results about L\'{e}vy processes in Hilbert spaces.
Let $U$ be a separable Hilbert space. 

\begin{definition}\label{def-Wiener-process}
An $U$-valued c\`{a}dl\`{a}g, adapted process $L$ is called a \emph{L\'{e}vy process} if the following conditions are satisfied:
\begin{enumerate}
\item We have $L_0 = 0$.

\item $L_t - L_s$ is independent of $\mathcal{F}_s$ for all $s \leq t$.

\item We have $L_t - L_s \overset{{\rm d}}{=} L_{t-s}$ for all $s \leq t$.
\end{enumerate}
\end{definition}

\begin{definition}
An $U$-valued L\'{e}vy process $L$ with $\mathbb{E}[\| L_t \|_U^2] < \infty$ and $\mathbb{E}[L_t] = 0$ for all $t \geq 0$ is called a \emph{square-integrable L\'{e}vy martingale}. 
\end{definition}

Note that any square-integrable L\'{e}vy martingale $L$ is indeed a martingale, that is
\begin{align*}
\mathbb{E}[X_t \,|\, \mathcal{F}_s] = X_s \quad \text{for all $s \leq t$,}
\end{align*}
see \cite[Prop.~3.25]{P-Z-book}. According to \cite[Thm.~4.44]{P-Z-book}, for each square-integrable L\'{e}vy martingale $L$ there exists a unique 
self-adjoint, nonnegative definite trace class operator $Q \in L(U)$, called the \emph{covariance operator} of $L$, such that for all $t,s \in \mathbb{R}_+$ and $u_1,u_2 \in U$ we have
\begin{align*}
\mathbb{E}[\langle L_t,u_1 \rangle_U \langle L_s,u_2 \rangle_U] = (t \wedge s) \langle Q u_1,u_2 \rangle_U.
\end{align*}
Moreover, for all $u_1,u_2 \in U$ the angle bracket process is given by
\begin{align}\label{angle-bracket-Levy}
\langle \langle L,u_1 \rangle_U, \langle L,u_2 \rangle_U \rangle_t = t \langle Qu_1,u_2 \rangle_U, \quad t \geq 0,
\end{align}
see \cite[Thm.~4.49]{P-Z-book}.

\begin{lemma}\label{lemma-Wiener-transform}
Let $L$ be an $U$-valued square-integrable L\'{e}vy martingale with covariance operator $Q$, let $V$ be another separable Hilbert space and let $\Phi : U \rightarrow V$ be an isometric isomorphism. Then the process $\Phi(L)$ is a $V$-valued square-integrable L\'{e}vy martingale with covariance operator $Q_{\Phi} := \Phi Q \Phi^{-1}$.
\end{lemma}

\begin{proof}
The process $\Phi(L)$ is a $V$-valued c\`{a}dl\`{a}g, adapted process with $\Phi(L_0) = \Phi(0) = 0$. Let $s \leq t$ be arbitrary. Then the random variable $\Phi(L_t) - \Phi(L_s) = \Phi(L_t - L_s)$ is independent of $\mathcal{F}_s$, and we have
\begin{align*}
\Phi(L_t) - \Phi(L_s) = \Phi(L_t - L_s) \overset{{\rm d}}{=} \Phi(L_{t-s}),
\end{align*}
Moreover, for each $t \in \mathbb{R}_+$ we have 
\begin{align*}
\mathbb{E}[ \| \Phi(L_t) \|_V^2] = \mathbb{E}[ \| L_t \|_U^2 ] < \infty \quad \text{and} \quad \mathbb{E}[\Phi(L_t)] = \Phi \mathbb{E}(L_t) = 0, 
\end{align*}
showing that $\Phi(L)$ is a $V$-valued square-integrable L\'{e}vy martingale.

Let $t,s \in \mathbb{R}_+$ and $v_i \in V$, $i=1,2$ be arbitrary, and set $u_i := \Phi^{-1} v_i \in U$, $i=1,2$. Then we have
\begin{align*}
&\mathbb{E}[\langle \Phi(L_t),v_1 \rangle_V \langle \Phi(L_s),v_2 \rangle_V] = \mathbb{E}[\langle \Phi(L_t),\Phi(u_1) \rangle_V \langle \Phi(L_s),\Phi(u_2) \rangle_V]
\\ &= \mathbb{E} [\langle L_t,u_1 \rangle_U \langle L_s,u_2 \rangle_U] = (t \wedge s) \langle Q u_1,u_2 \rangle_U = (t \wedge s) \langle Q \Phi^{-1} v_1,\Phi^{-1} v_2 \rangle_U
\\ &= (t \wedge s) \langle \Phi Q \Phi^{-1} v_1,v_2 \rangle_V = (t \wedge s) \langle Q_{\Phi} v_1,v_2 \rangle_V,
\end{align*}
showing that the L\'{e}vy martingale $\Phi(L)$ has the covariance operator $Q_{\Phi}$.
\end{proof}

Now, let $Q \in L(U)$ be a self-adjoint, positive definite trace class operator. Then there exist a sequence $(\lambda_j)_{j \in \mathbb{N}} \subset (0,\infty)$ with $\sum_{j=1}^{\infty} \lambda_j < \infty$ and an orthonormal basis $(e_j^{(\lambda)})_{j \in \mathbb{N}}$ of $U$ and such that
\begin{align*}
Q e_j^{(\lambda)} = \lambda_j e_j^{(\lambda)} \quad \text{for all $j \in \mathbb{N}$.}
\end{align*}
We define the sequence of pairwise orthogonal vectors $(e_j)_{j \in \mathbb{N}}$ as
\begin{align*}
e_j := \sqrt{\lambda_j} e_j^{(\lambda)}, \quad j \in \mathbb{N}.
\end{align*}

\begin{proposition}\label{prop-W-beta-j}
Let $L$ be an $U$-valued square-integrable L\'{e}vy martingale with covariance operator $Q$. Then the sequence $(M^j)_{j \in \mathbb{N}}$ given by
\begin{align}\label{def-beta-j}
M^j := \frac{1}{\sqrt{\lambda_j}} \langle L,e_j^{(\lambda)} \rangle_U, \quad j \in \mathbb{N}.
\end{align}
is a sequence of standard L\'{e}vy processes.
\end{proposition}

\begin{proof}
For each $j \in \mathbb{N}$ the process $M^j$ is a real-valued square-integrable L\'{e}vy martingale. By (\ref{angle-bracket-Levy}), for all $j,k \in \mathbb{N}$ we obtain
\begin{align*}
\langle M^j,M^k \rangle_t &= \frac{1}{\sqrt{\lambda_j \lambda_k}} \langle \langle L,e_j^{(\lambda)} \rangle_U, \langle L,e_k^{(\lambda)} \rangle_U \rangle_t = \frac{t \langle Q e_j^{(\lambda)}, e_k^{(\lambda)} \rangle_U}{\sqrt{\lambda_j \lambda_k}}
\\ &= \frac{t \lambda_j \langle e_j^{(\lambda)}, e_k^{(\lambda)} \rangle_U}{\sqrt{\lambda_j \lambda_k}} = \delta_{jk} \cdot t,
\end{align*}
showing that $(M^j)_{j \in \mathbb{N}}$ is a sequence of standard L\'{e}vy processes.
\end{proof}

\section{The It\^{o} integral with respect to a $\ell_{\lambda}^2$-valued L\'{e}vy process}\label{sec-step-3}

In this section, we introduce the It\^{o} integral for $\ell^2(H)$-valued processes with respect to a $\ell_{\lambda}^2$-valued L\'{e}vy process, which is based on the It\^{o} integral (\ref{series-int-2}) from Section~\ref{sec-step-2}.

Let $(\lambda_j)_{j \in \mathbb{N}} \subset (0,\infty)$ be a sequence with $\sum_{j = 1}^{\infty} \lambda_j < \infty$ and denote by $\ell_{\lambda}^2$ the weighted space of sequences
\begin{align*}
\ell_{\lambda}^2 := \bigg\{ (v^j)_{j \in \mathbb{N}} \subset \mathbb{R} : \sum_{j=1}^{\infty} \lambda_j|v^j|^2 < \infty \bigg\},
\end{align*}
which, equipped with the inner product
\begin{align*}
\langle v,w \rangle_{\ell_{\lambda}^2} = \sum_{j=1}^{\infty} \lambda_j v^j w^j
\end{align*}
is a separable Hilbert space. Note that we have the strict inclusion $\ell^2 \subsetneqq \ell_{\lambda}^2$, where $\ell^2$ denotes the space of sequences
\begin{align*}
\ell^2 = \bigg\{ (v^j)_{j \in \mathbb{N}} \subset \mathbb{R} : \sum_{j=1}^{\infty} |v^j|^2 < \infty \bigg\}.
\end{align*}
We denote by $(g_j)_{j \in \mathbb{N}}$ the standard orthonormal basis of $\ell^2$, which is given by
\begin{align*}
g_1 = (1,0,\ldots), \quad g_2 = (0,1,0,\ldots), \quad \ldots
\end{align*}
Then the system $(g_j^{(\lambda)})_{j \in \mathbb{N}}$ defined as
\begin{align}\label{ONB-g}
g_j^{(\lambda)} := \frac{g_j}{\sqrt{\lambda_j}}, \quad j \in \mathbb{N}
\end{align}
is an orthonormal basis of $\ell_{\lambda}^2$. Let $Q \in L(\ell_{\lambda}^2)$ be a linear operator such that
\begin{align}\label{eigenvalues-l2-lambda}
Q g_j^{(\lambda)} = \lambda_j g_j^{(\lambda)} \quad \text{for all $j \in \mathbb{N}$.}
\end{align}
Then $Q$ is a nuclear, self-adjoint, positive definite operator. Let $L$ be a $\ell_{\lambda}^2$-valued, square-integrable L\'{e}vy martingale with covariance operator $Q$. According to Proposition~\ref{prop-W-beta-j}, the sequence $(M^j)_{j \in \mathbb{N}}$ given by
\begin{align*}
M^j := \frac{1}{\sqrt{\lambda_j}} \langle L,g_j^{(\lambda)} \rangle_{\ell_{\lambda}^2}, \quad j \in \mathbb{N}
\end{align*}
is a sequence of standard L\'{e}vy processes.

\begin{definition}\label{def-Ito-step-2}
For every $\ell^2(H)$-valued, predictable process $X$ satisfying (\ref{int-cond-2}) we define the \emph{It\^{o} integral} $X \bullet L := (\int_0^t X_s dL_s)_{t \in [0,T]}$ as
\begin{align}\label{def-as-series}
X \bullet L := \sum_{j \in \mathbb{N}} X^j \bullet M^j.
\end{align}
\end{definition}

\begin{remark}
Note that $L_2^0(H) \cong \ell^2(H)$, where $L_2^0(H)$ denotes the space of Hilbert-Schmidt operators from $\ell^2$ to $H$. In \cite{fillnm}, the It\^{o} integral for $L_2^0(H)$-valued processes with respect to a $\ell_{\lambda}^2$-valued Wiener process has been constructed in the usual fashion (first for elementary and afterwards for general processes), and then the series representation (\ref{def-as-series}) has been proven, see \cite[Prop.~2.2.1]{fillnm}.
\end{remark}

Now, let $(\mu_k)_{k \in \mathbb{N}}$ be another sequence with $\sum_{k=1}^{\infty} \mu_k < \infty$, and let $\Phi : \ell_{\lambda}^2 \rightarrow \ell_{\mu}^2$ be an isometric isomorphism such that
\begin{align}\label{eigenvalues-l2-mu}
Q_{\Phi} g_k^{(\mu)} = \mu_k g_k^{(\mu)} \quad \text{for all $k \in \mathbb{N}$.}
\end{align}
By Lemma~\ref{lemma-Wiener-transform}, the process $\Phi(L)$ is a $\ell_{\mu}^2$-valued, square integrable L\'{e}vy martingale with covariance operator $Q_{\Phi}$, and by Proposition~\ref{prop-W-beta-j}, the sequence $(N^k)_{k \in \mathbb{N}}$ given by
\begin{align*}
N^k := \frac{1}{\sqrt{\mu_k}} \langle \Phi(L),g_k^{(\mu)} \rangle_{\ell_{\mu}^2}, \quad k \in \mathbb{N}
\end{align*}
is a sequence of standard L\'{e}vy processes.

\begin{theorem}\label{thm-angle}
Let $\Psi \in L(\ell^2(H))$ be an isometric isomorphism such that
\begin{align}\label{rel-Phi-Psi}
\langle h,\Psi(w) \rangle_H = \Phi( \langle h,w \rangle_H) \quad \text{for all $h \in H$ and $w \in \ell^2(H)$.}
\end{align}
Then for every $\ell^2(H)$-valued, predictable process $X$ satisfying (\ref{int-cond-2}) we have
\begin{align}\label{Psi-l2-integrable}
\mathbb{E} \bigg[ \int_0^T \| \Psi(X_s) \|_{\ell^2(H)}^2 ds \bigg] < \infty
\end{align}
and the identity
\begin{align}\label{identity-geom-angle}
X \bullet L = \Psi(X) \bullet \Phi(L).
\end{align}
\end{theorem}

\begin{proof}
Since $\Psi$ is an isometry, by (\ref{int-cond-2}) we have
\begin{align*}
\mathbb{E} \bigg[ \int_0^T \| \Psi(X_s) \|_{\ell^2(H)}^2 ds \bigg] = \mathbb{E} \bigg[ \int_0^T \| X_s \|_{\ell^2(H)}^2 ds \bigg] < \infty,
\end{align*}
showing (\ref{Psi-l2-integrable}). Moreover, by (\ref{eigenvalues-l2-mu}) we have
\begin{align*}
\Phi Q \Phi^{-1} g_k^{(\mu)} = Q_{\Phi} g_k^{(\mu)} = \mu_k g_k^{(\mu)} \quad \text{for all $k \in \mathbb{N}$,}
\end{align*}
and hence, we get
\begin{align}\label{eigenvalues-l2-mu-x}
Q (\Phi^{-1} g_k^{(\mu)}) = \mu_k (\Phi^{-1} g_k^{(\mu)}) \quad \text{for all $k \in \mathbb{N}$.}
\end{align}
By (\ref{eigenvalues-l2-lambda}) and (\ref{eigenvalues-l2-mu-x}), the vectors $(g_j^{(\lambda)})_{j \in \mathbb{N}}$ and $(\Phi^{-1} g_k^{(\mu)})_{k \in \mathbb{N}}$ are eigenvectors of $Q$ with corresponding eigenvalues $(\lambda_j)_{j \in \mathbb{N}}$ and $(\mu_k)_{k \in \mathbb{N}}$. Therefore, and since $\Phi$ is an isometry, for $j,k \in \mathbb{N}$ with $\lambda_j \neq \mu_k$ we obtain
\begin{align}\label{orth-eigenvectors}
\langle \Phi g_j^{(\lambda)}, g_k^{(\mu)} \rangle_{\ell_{\mu}^2} = \langle g_j^{(\lambda)},\Phi^{-1} g_k^{(\mu)} \rangle_{\ell_{\lambda}^2} = 0.
\end{align}
Let $h \in H$ be arbitrary. Then we have
\begin{align*}
&\langle h,X \bullet L \rangle_H = \Big\langle h, \sum_{j=1}^{\infty} X^j \bullet M^j \Big\rangle_H = \sum_{j=1}^{\infty} \langle h,X^j \bullet M^j \rangle_H = \sum_{j=1}^{\infty} \langle h,X^j \rangle_H \bullet M^j
\\ &= \sum_{j=1}^{\infty} \frac{1}{\sqrt{\lambda_j}} \langle \langle h,X \rangle_H, g_j^{(\lambda)} \rangle_{\ell_{\lambda}^2} \bullet \frac{1}{\sqrt{\lambda_j}} \langle L,g_j^{(\lambda)} \rangle_{\ell_{\lambda}^2}
\\ &= \sum_{j=1}^{\infty} \frac{1}{\sqrt{\lambda_j}} \langle \Phi(\langle h,X \rangle_H), \Phi g_j^{(\lambda)} \rangle_{\ell_{\mu}^2} \bullet \frac{1}{\sqrt{\lambda_j}} \langle \Phi(L),\Phi g_j^{(\lambda)} \rangle_{\ell_{\mu}^2}
\\ &= \sum_{j=1}^{\infty} \frac{1}{\sqrt{\lambda_j}} \langle \Phi(\langle h,X \rangle_H), \Phi g_j^{(\lambda)} \rangle_{\ell_{\mu}^2} \bullet \frac{1}{\sqrt{\lambda_j}} \bigg( \sum_{k=1}^{\infty} \langle \Phi(L), g_k^{(\mu)} \rangle_{\ell_{\mu}^2} \langle g_k^{(\mu)},\Phi g_j^{(\lambda)} \rangle_{\ell_{\mu}^2} \bigg).
\end{align*}
Since $(\lambda_j)_{j \in \mathbb{N}}$ and $(\mu_k)_{k \in \mathbb{N}}$ are eigenvalues of $Q$, for each $j \in \mathbb{N}$ there are only finitely many $k \in \mathbb{N}$ such that $\lambda_j = \mu_k$. Therefore, by (\ref{orth-eigenvectors}), and since
$(\Phi(g_j^{(\lambda)}))_{j \in \mathbb{N}}$ is an orthonormal basis of $\ell_{\mu}^2$, we obtain
\begin{align*}
&\langle h,X \bullet L \rangle_H 
\\ &= \sum_{k=1}^{\infty} \bigg( \sum_{j=1}^{\infty} \frac{1}{\lambda_j} \langle \Phi(\langle h,X \rangle_H), \Phi g_j^{(\lambda)} \rangle_{\ell_{\mu}^2} \langle \Phi g_j^{(\lambda)},g_k^{(\mu)} \rangle_{\ell_{\mu}^2} \bigg) \bullet \langle \Phi(L), g_k^{(\mu)} \rangle_{\ell_{\mu}^2}
\\ &= \sum_{k=1}^{\infty} \frac{1}{\mu_k} \bigg( \sum_{j=1}^{\infty} \langle \Phi(\langle h,X \rangle_H), \Phi g_j^{(\lambda)} \rangle_{\ell_{\mu}^2} \langle \Phi g_j^{(\lambda)},g_k^{(\mu)} \rangle_{\ell_{\mu}^2} \bigg) \bullet \langle \Phi(L), g_k^{(\mu)} \rangle_{\ell_{\mu}^2}
\\ &= \sum_{k=1}^{\infty} \frac{1}{\sqrt{\mu_k}} \langle \Phi(\langle h,X \rangle_H),g_k^{(\mu)} \rangle_{\ell_{\mu}^2} \bullet \frac{1}{\sqrt{\mu_k}} \langle \Phi(L), g_k^{(\mu)} \rangle_{\ell_{\mu}^2}
= \sum_{k=1}^{\infty} \Phi(\langle h,X \rangle_H)^k \bullet N^k.
\end{align*}
Thus, taking into account (\ref{rel-Phi-Psi}) gives us
\begin{align*}
\langle h,X \bullet L \rangle_H &= \sum_{k=1}^{\infty} \langle h,\Psi(X)^k \rangle_H \bullet N^k = \sum_{k=1}^{\infty} \langle h, \Psi(X)^k \bullet N^k \rangle_H
\\ &= \Big\langle h, \sum_{k=1}^{\infty} \Psi(X)^k \bullet N^k \Big\rangle_H = \langle h, \Psi(X) \bullet \Phi(L) \rangle_H.
\end{align*}
Since $h \in H$ was arbitrary, using the separability of $H$ as in the proof of Proposition~\ref{prop-int-well-def},
we arrive at (\ref{identity-geom-angle}).
\end{proof}

\begin{remark}
From a geometric point of view, Theorem~\ref{thm-angle} says that the ``angle'' measured by the It\^{o} integral is preserved under isometries.
\end{remark}

\section{The It\^{o} integral with respect to a general L\'{e}vy process}\label{sec-step-4}

In this section, we define the It\^{o} integral with respect to a general L\'{e}vy process, which is based on the It\^{o} integral (\ref{def-as-series}) from the previous section.

Let $U$ be a separable Hilbert space and let $Q \in L(U)$ be a nuclear, self-adjoint, positive definite linear operator. Then there exist a sequence $(\lambda_j)_{j \in \mathbb{N}} \subset (0,\infty)$ with $\sum_{j=1}^{\infty} \lambda_j < \infty$ and an orthonormal basis $(e_j^{(\lambda)})_{j \in \mathbb{N}}$ of $U$ and such that
\begin{align}\label{Q-e-diagonal}
Q e_j^{(\lambda)} = \lambda_j e_j^{(\lambda)} \quad \text{for all $j \in \mathbb{N}$,}
\end{align}
namely, the $\lambda_j$ are the eigenvalues of $Q$, and each $e_j^{(\lambda)}$
is an eigenvector corresponding to $\lambda_j$. The space $U_0 := Q^{1/2}(U)$, equipped with the inner product
\begin{align*}
\langle u,v \rangle_{U_0} := \langle Q^{-1/2} u, Q^{-1/2} v \rangle_{U}, 
\end{align*}
is another separable Hilbert space
and the sequence $(e_j)_{j \in \mathbb{N}}$ given by
\begin{align*}
e_j = \sqrt{\lambda_j} e_j^{(\lambda)}, \quad j \in \mathbb{N}
\end{align*}
is an orthonormal basis of $U_0$. We denote by $L_2^0(H) := L_2(U_0,H)$ the space of Hilbert-Schmidt
operators from $U_0$ into $H$, which, endowed with the
Hilbert-Schmidt norm
\begin{align*}
\| S \|_{L_2^0(H)} := \bigg( \sum_{j=1}^{\infty} \| S e_j \|_H^2 \bigg)^{1/2},
\quad S \in L_2^0(H)
\end{align*}
itself is a separable Hilbert space. We define the isometric isomorphisms 
\begin{align}\label{Phi-lambda}
&\Phi_{\lambda} : U \rightarrow \ell_{\lambda}^2, \quad \Phi_{\lambda} e_j^{(\lambda)} := g_j^{(\lambda)} \text{ for $j \in \mathbb{N}$,}
\\ \label{Psi-lambda} &\Psi_{\lambda} : L_2^0(H) \rightarrow \ell^2(H), \quad \Psi_{\lambda}(S) := \big( S e_j \big)_{j \in \mathbb{N}} \text{ for $S \in L_2^0(H)$.} 
\end{align}
Recall that $(g_j^{(\lambda)})_{j \in \mathbb{N}}$ denotes the orthonormal basis of $\ell_{\lambda}^2$, which we have defined in (\ref{ONB-g}). Let $L$ be an $U$-valued square-integrable L\'{e}vy martingale with covariance operator $Q$.

\begin{lemma}\label{lemma-transform-l2}
The following statements are true:
\begin{enumerate}
\item The process $\Phi_{\lambda}(L)$ is a $\ell_{\lambda}^2$-valued square-integrable L\'{e}vy martingale with covariance operator $Q_{\Phi_{\lambda}}$.

\item We have
\begin{align}\label{Q-trans-diag}
Q_{\Phi_{\lambda}} g_j^{(\lambda)} = \lambda_j g_j^{(\lambda)} \quad \text{for all $j \in \mathbb{N}$.}
\end{align}
\end{enumerate}
\end{lemma}

\begin{proof}
By Lemma~\ref{lemma-Wiener-transform}, the process $\Phi_{\lambda}(L)$ is a $\ell_{\lambda}^2$-valued square-integrable L\'{e}vy martingale with covariance operator $Q_{\Phi_{\lambda}}$. Furthermore, by (\ref{Phi-lambda}) and (\ref{Q-e-diagonal}), for all $j \in \mathbb{N}$ we obtain
\begin{align*}
Q_{\Phi_{\lambda}} g_j^{(\lambda)} = \Phi_{\lambda} Q \Phi_{\lambda}^{-1} g_j^{(\lambda)} = \Phi_{\lambda} Q e_j^{(\lambda)} = \Phi_{\lambda}(\lambda_j e_j^{(\lambda)}) = \lambda_j \Phi_{\lambda} e_j^{(\lambda)} = \lambda_j g_j^{(\lambda)},
\end{align*}
showing (\ref{Q-trans-diag}).
\end{proof}

Now, our idea is to the define the It\^{o} integral for a $L_2^0(H)$-valued, predictable process $X$ with
\begin{align}\label{int-cond-3}
\mathbb{E} \bigg[ \int_0^T \| X_s \|_{L_2^0(H)}^2 ds \bigg] < \infty
\end{align}
by setting
\begin{align}\label{def-integral-general}
X \bullet L := \Psi_{\lambda}(X) \bullet \Phi_{\lambda}(L),
\end{align}
where the right-hand side of (\ref{def-integral-general}) denotes the It\^{o} integral (\ref{def-as-series}) from Definition~\ref{def-Ito-step-2}.
One might suspect that this definition depends on the choice of the eigenvalues $(\lambda_j)_{j \in \mathbb{N}}$ and eigenvectors $(e_j^{(\lambda)})_{j \in \mathbb{N}}$. In order to prove that this is not the case,
let $(\mu_k)_{k \in \mathbb{N}} \subset (0,\infty)$ be another sequence with $\sum_{k=1}^{\infty} \mu_k < \infty$ and let $(f_k^{(\mu)})_{k \in \mathbb{N}}$ be another orthonormal basis of $U$ such that
\begin{align}\label{Q-f-diagonal}
Q f_k^{(\mu)} = \mu_k f_k^{(\mu)} \quad \text{for all $k \in \mathbb{N}$.}
\end{align}
Then the sequence $(f_k)_{k \in \mathbb{N}}$ given by
\begin{align*}
f_k = \sqrt{\mu_k} f_k^{(\mu)}, \quad k \in \mathbb{N}
\end{align*}
is an orthonormal basis of $U_0$. Analogous to (\ref{Phi-lambda}) and (\ref{Psi-lambda}), we define the isometric isomorphisms 
\begin{align*}
&\Phi_{\mu} : U \rightarrow \ell_{\mu}^2, \quad \Phi_{\mu} f_k^{(\mu)} := g_k^{(\mu)} \text{ for $k \in \mathbb{N}$,}
\\ &\Psi_{\mu} : L_2^0(H) \rightarrow \ell^2(H), \quad \Psi_{\mu}(S) := \big( S f_k \big)_{k \in \mathbb{N}} \text{ for $S \in L_2^0(H)$.} 
\end{align*}
Furthermore, we define the isometric isomorphisms
\begin{align*}
\Phi := \Phi_{\mu} \circ \Phi_{\lambda}^{-1} : \ell_{\lambda}^2 \rightarrow \ell_{\mu}^2 \quad \text{and} \quad \Psi := \Psi_{\mu} \circ \Psi_{\lambda}^{-1} : \ell^2(H) \rightarrow \ell^2(H).
\end{align*}
The following diagram illustrates the situation:
\begin{align*}
\begin{xy}
  \xymatrix{
      (\ell_{\lambda}^2,\ell^2(H)) \ar@<2pt>[rr]^{(\Phi,\Psi)}  &     &  (\ell_{\mu}^2,\ell^2(H))\\
                             & \ar[ul]_{(\Phi_{\lambda},\Psi_{\lambda})} (U, L_2^0(H)) \ar[ru]^{(\Phi_{\mu},\Psi_{\mu})} &
  }
\end{xy}
\end{align*}
In order to show that the It\^{o} integral (\ref{def-integral-general}) is well-defined, we have to show that
\begin{align}\label{indentity-angle}
\Psi_{\lambda}(X) \bullet \Phi_{\lambda}(L) = \Psi_{\mu}(X) \bullet \Phi_{\mu}(L).
\end{align}
For this, we prepare the following auxiliary result:

\begin{lemma}\label{lemma-Phi-Psi}
For all $h \in H$ and $w \in \ell^2(H)$ we have
\begin{align*}
\langle h,\Psi(w) \rangle_H = \Phi(\langle h,w \rangle_H).
\end{align*}
\end{lemma}

\begin{proof}
By (\ref{Q-e-diagonal}) and (\ref{Q-f-diagonal}), the vectors $(e_j^{(\lambda)})_{j \in \mathbb{N}}$ and $(f_k^{(\mu)})_{k \in \mathbb{N}}$ are eigenvectors of $Q$ with corresponding eigenvalues $(\lambda_j)_{j \in \mathbb{N}}$ and $(\mu_k)_{k \in \mathbb{N}}$. Therefore, for $j,k \in \mathbb{N}$ with $\lambda_j \neq \mu_k$ we have $\langle e_j^{(\lambda)},f_k^{(\mu)} \rangle_U = 0$.
For each $v \in \ell_{\lambda}^2$ we obtain
\begin{align*}
\Phi(v) &= (\Phi_{\mu} \circ \Phi_{\lambda}^{-1})(v) = \Phi_{\mu} \bigg( \sum_{k=1}^{\infty} \langle \Phi_{\lambda}^{-1} v,f_k^{(\mu)} \rangle_U \, f_k^{(\mu)} \bigg) = \sum_{k=1}^{\infty} \langle \Phi_{\lambda}^{-1} v,f_k^{(\mu)} \rangle_U \, g_k^{(\mu)}
\\ &= \bigg( \frac{1}{\sqrt{\mu_k}} \langle \Phi_{\lambda}^{-1} v,f_k^{(\mu)} \rangle_U \bigg)_{k \in \mathbb{N}} = \bigg( \frac{1}{\sqrt{\mu_k}} \sum_{j=1}^{\infty} \langle \Phi_{\lambda}^{-1} v, e_j^{(\lambda)} \rangle_U \langle e_j^{(\lambda)}, f_k^{(\mu)} \rangle_U \bigg)_{k \in \mathbb{N}}
\\ &= \bigg( \frac{1}{\sqrt{\mu_k}} \sum_{j=1}^{\infty} \langle v, g_j^{(\lambda)} \rangle_{\ell_{\lambda}^2} \langle e_j^{(\lambda)}, f_k^{(\mu)} \rangle_U \bigg)_{k \in \mathbb{N}}
\\ &= \bigg( \sum_{j=1}^{\infty} \frac{\sqrt{\lambda_j}}{\sqrt{\mu_k}} \langle e_j^{(\lambda)}, f_k^{(\mu)} \rangle_U \, v^j \bigg)_{k \in \mathbb{N}} = \bigg( \sum_{j=1}^{\infty} \langle e_j^{(\lambda)}, f_k^{(\mu)} \rangle_U \, v^j \bigg)_{k \in \mathbb{N}}.
\end{align*}
Let $w \in \ell^2(H)$ be arbitrary. By (\ref{Psi-lambda}) we have
\begin{align*}
w = \Psi_{\lambda}(\Psi_{\lambda}^{-1}(w)) = \big( \Psi_{\lambda}^{-1}(w) e_j \big)_{j \in \mathbb{N}},
\end{align*}
and hence
\begin{align*}
\Psi(w) &= (\Psi_{\mu} \circ \Psi_{\lambda}^{-1})(w) = \big( \Psi_{\lambda}^{-1}(w)f_k \big)_{k \in \mathbb{N}} = \bigg( \Psi_{\lambda}^{-1}(w) \bigg( \sum_{j=1}^{\infty} \langle f_k,e_j \rangle_{U_0} e_j \bigg) \bigg)_{k \in \mathbb{N}}
\\ &= \bigg( \sum_{j=1}^{\infty} \langle f_k,e_j \rangle_{U_0} \Psi_{\lambda}^{-1}(w) e_j \bigg)_{k \in \mathbb{N}} = \bigg( \sum_{j=1}^{\infty} \langle f_k,e_j \rangle_{U_0} w^j \bigg)_{k \in \mathbb{N}} 
\\ &= \bigg( \sum_{j=1}^{\infty} \langle e_j^{(\lambda)}, f_k^{(\mu)} \rangle_U \, w^j \bigg)_{k \in \mathbb{N}}.
\end{align*}
Therefore, for all $h \in H$ and $w \in \ell^2(H)$ we obtain
\begin{align*}
\langle h,\Psi(w) \rangle_H = \bigg( \sum_{j=1}^{\infty} \langle e_j^{(\lambda)},f_k^{(\mu)} \rangle_U \langle h,w^j \rangle_H \bigg)_{k \in \mathbb{N}} = \Phi(\langle h,w \rangle_H),
\end{align*}
finishing the proof.
\end{proof}

\begin{proposition}\label{prop-well-defined}
The following statements are true:
\begin{enumerate}
\item $\Phi_{\lambda}(L)$ is a $\ell_{\lambda}^2$-valued L\'{e}vy process with covariance operator $Q_{\Phi_{\lambda}}$, and we have 
\begin{align*}
Q_{\Phi_{\lambda}} g_j^{(\lambda)} = \lambda_j g_j^{(\lambda)} \quad \text{for all $j \in \mathbb{N}$.}
\end{align*}

\item $\Phi_{\mu}(L)$ is a $\ell_{\mu}^2$-valued L\'{e}vy process with covariance operator $Q_{\Phi_{\mu}}$, and we have
\begin{align*}
Q_{\Phi_{\mu}} g_k^{(\mu)} = \mu_k g_k^{(\mu)} \quad \text{for all $k \in \mathbb{N}$.}
\end{align*}

\item For every $L_2^0(H)$-valued, predictable process $X$ with (\ref{int-cond-3}) we have
\begin{align}\label{int-well-exists}
\mathbb{E} \bigg[ \int_0^T \| \Psi_{\lambda}(X_s) \|_{\ell^2(H)}^2 ds \bigg] < \infty, \quad \mathbb{E} \bigg[ \int_0^T \| \Psi_{\mu}(X_s) \|_{\ell^2(H)}^2 ds \bigg] < \infty
\end{align}
and the identity (\ref{indentity-angle}).
\end{enumerate}
\end{proposition}

\begin{proof}
The first two statements follow from Lemma~\ref{lemma-transform-l2}.
Since $\Psi_{\lambda}$ and $\Psi_{\mu}$ are isometries, we obtain
\begin{align*}
\mathbb{E} \bigg[ \int_0^T \| \Psi_{\lambda}(X_s) \|_{\ell^2(H)}^2 ds \bigg] = \mathbb{E} \bigg[ \int_0^T \| X_s \|_{L_2^0(H)}^2 ds \bigg] = \mathbb{E} \bigg[ \int_0^T \| \Psi_{\mu}(X_s) \|_{\ell^2(H)}^2 ds \bigg],
\end{align*}
which, together with (\ref{int-cond-3}), yields (\ref{int-well-exists}).
Now, Theorem~\ref{thm-angle} applies by virtue of Lemma~\ref{lemma-Phi-Psi}, and yields
\begin{align*}
\Psi_{\lambda}(X) \bullet \Phi_{\lambda}(L) = \Psi(\Psi_{\lambda}(X)) \bullet \Phi(\Phi_{\lambda}(L)) = \Psi_{\mu}(X) \bullet \Phi_{\mu}(L),
\end{align*}
proving (\ref{indentity-angle}).
\end{proof}

\begin{definition}\label{def-Ito-step-3}
For every $L_2^0(H)$-valued process $X$ satisfying (\ref{int-cond-3})
we define the \emph{It\^{o}-Integral} $X \bullet L = (\int_0^t X_s dL_s)_{t \in [0,T]}$ by (\ref{def-integral-general}).
\end{definition}

By virtue of Proposition~\ref{prop-well-defined}, the Definition (\ref{def-integral-general}) of the It\^{o} integral neither depends on the choice of the eigenvalues $(\lambda_j)_{j \in \mathbb{N}}$ nor on the eigenvectors $(e_j^{(\lambda)})_{j \in \mathbb{N}}$.

Now, we shall prove the announced series representation of the It\^{o} integral.
According to Proposition~\ref{prop-W-beta-j}, the sequences $(M^j)_{j \in \mathbb{N}}$ and $(N^j)_{j \in \mathbb{N}}$ of real-valued processes given by
\begin{align*}
M^j := \frac{1}{\sqrt{\lambda_j}} \langle L,e_j^{(\lambda)} \rangle_U \quad \text{and} \quad N^j := \frac{1}{\sqrt{\lambda_j}} \langle \Phi_{\lambda}(L), g_j^{(\lambda)} \rangle_{\ell_{\lambda}^2}
\end{align*}
are sequences of standard L\'{e}vy processes.

\begin{proposition}\label{prop-int-as-sum}
For every $L_2^0(H)$-valued, predictable process $X$ satisfying (\ref{int-cond-3}) the process $(\xi^j)_{j \in \mathbb{N}}$ given by
\begin{align*}
\xi^j := X e_j, \quad j \in \mathbb{N}
\end{align*}
is a $\ell^2(H)$-valued, predictable process, and we have
\begin{align}\label{int-as-sum}
X \bullet L = \sum_{j \in \mathbb{N}} \xi^j \bullet M^j,
\end{align}
where the right-hand side of (\ref{int-as-sum}) converges unconditionally in $M_T^2(H)$.
\end{proposition}

\begin{proof}
Since $\Phi_{\lambda}$ is an isometry, for each $j \in \mathbb{N}$ we obtain
\begin{align*}
M^j = \frac{1}{\sqrt{\lambda_j}} \langle L,e_j^{(\lambda)} \rangle_U = \frac{1}{\sqrt{\lambda_j}} \langle \Phi_{\lambda}(L), \Phi_{\lambda} e_j^{(\lambda)} \rangle_{\ell_{\lambda}^2} = \frac{1}{\sqrt{\lambda_j}} \langle \Phi_{\lambda}(L),g_j^{(\lambda)} \rangle_{\ell_{\lambda}^2} = N^j.
\end{align*}
Thus, by Definitions~\ref{def-Ito-step-3} and \ref{def-Ito-step-2} we obtain
\begin{align*}
X \bullet L = \Psi_{\lambda}(X) \bullet \Phi_{\lambda}(L) = \sum_{j \in \mathbb{N}} \Psi_{\lambda}(X)^j \bullet N^j = \sum_{j \in \mathbb{N}} \xi^j \bullet M^j
\end{align*}
and, by Proposition~\ref{prop-int-conv-2}, the series converges unconditionally in $M_T^2(H)$.
\end{proof}

\begin{remark}
By Remark~\ref{rem-orth} and the proof of Proposition~\ref{prop-int-as-sum}, the components of the It\^{o} integral $\sum_{j \in \mathbb{N}} \xi^j \bullet M^j$ are pairwise orthogonal elements of the Hilbert space $M_T^2(H)$.
\end{remark}

\begin{proposition}
For every $L_2^0(H)$-valued process $X$ satisfying (\ref{int-cond-3}) we have the \emph{It\^{o} isometry}
\begin{align*}
\mathbb{E} \Bigg[ \bigg\| \int_0^T X_s dL_s \bigg\|_H^2 \Bigg] = \mathbb{E} \bigg[ \int_0^T \| L_s \|_{L_2^0(H)}^2 ds \bigg].
\end{align*}
\end{proposition}

\begin{proof}
By the It\^{o} isometry (Proposition~\ref{prop-Ito-isom-step-2}), and since $\Psi_{\lambda}$ is an isometry, we obtain
\begin{align*}
&\mathbb{E} \Bigg[ \bigg\| \int_0^T X_s dL_s \bigg\|_H^2 \Bigg] = \mathbb{E} \Bigg[ \bigg\| \int_0^T \Psi_{\lambda}(X_s) d\Phi_{\lambda}(L)_s \bigg\|_H^2 \Bigg] 
\\ &= \mathbb{E} \bigg[ \int_0^T \| \Psi_{\lambda}(X_s) \|_{\ell^2(H)}^2 ds \bigg] = \mathbb{E} \Bigg[ \int_0^T \| X_s \|_{L_2^0(H)}^2 ds \Bigg],
\end{align*}
completing the proof.
\end{proof}

We shall now prove that the stochastic integral, which we have defined so far, coincides with the It\^{o} integral developed in \cite{P-Z-book}. For this purpose, it suffices to consider elementary processes.
Note that for each operator $S \in L(U,H)$ the restriction $S|_{U_0}$ belongs to $L_2^0(H)$, because
\begin{align*}
\sum_{j=1}^{\infty} \| S e_j \|_{H}^2 &\leq \sum_{j=1}^{\infty} \| S \|_{L(U,H)}^2 \| e_j \|_U^2 = \| S \|_{L(U,H)}^2 \sum_{j=1}^{\infty} \| \sqrt{\lambda_j} e_j^{(\lambda)} \|_{U_0}^2 
\\ &= \| S \|_{L(U,H)}^2 \sum_{j=1}^{\infty} \lambda_j < \infty.
\end{align*}

\begin{proposition}
Let $X$ be a $L(U,H)$-valued simple process of the form
\begin{align*}
X = X_0 \mathbbm{1}_{\{0\}} + \sum_{i=1}^n X_i \mathbbm{1}_{(t_i,t_{i+1}]}
\end{align*}
with $0 = t_1 < \ldots < t_{n+1} = T$ and $\mathcal{F}_{t_i}$-measurable random variables $X_i : \Omega \rightarrow L(U,H)$ for $i = 0,\ldots,n$. Then we have
\begin{align*}
X|_{U_0} \bullet L = \sum_{i=1}^n X_i ( L^{t_{i+1}} - L^{t_i} ).
\end{align*}
\end{proposition}

\begin{proof}
The process $\Psi_{\lambda}(X|_{U_0})$ is a $\ell^2(H)$-valued simple process having the representation
\begin{align*}
\Psi_{\lambda}(X|_{U_0}) = \Psi_{\lambda}(X_0|_{U_0}) \mathbbm{1}_{\{0\}} + \sum_{i=1}^n \Psi_{\lambda}(X_i|_{U_0}) \mathbbm{1}_{(t_i,t_{i+1}]}.
\end{align*}
Thus, by Proposition~\ref{prop-simple-2}, and since $\Phi_{\lambda}$ is an isometry, we obtain
\begin{align*}
&X|_{U_0} \bullet L = \Psi_{\lambda}(X|_{U_0}) \bullet \Phi_{\lambda}(L) = \sum_{i=1}^n \sum_{j \in \mathbb{N}} \Psi_{\lambda}(X_i|_{U_0})^j \big( (N^j)^{t_{i+1}} - (N^j)^{t_i} \big)
\\ &= \sum_{i=1}^n \sum_{j \in \mathbb{N}} X_i e_j \frac{\langle \Phi_{\lambda}(L^{t_{i+1}} - L^{t_i}), g_j^{(\lambda)} \rangle_{\ell_{\lambda}^2}}{\sqrt{\lambda_j}} 
\\ &= \sum_{i=1}^n \sum_{j \in \mathbb{N}} X_i e_j^{(\lambda)} \langle L^{t_{i+1}} - L^{t_i}, \Phi_{\lambda}^{-1} g_j^{(\lambda)} \rangle_U
= \sum_{i=1}^n \sum_{j \in \mathbb{N}} X_i e_j^{(\lambda)} \langle L^{t_{i+1}} - L^{t_i}, e_j^{(\lambda)} \rangle_U
\\ &= \sum_{i=1}^n X_i \bigg( \sum_{j \in \mathbb{N}} \langle L^{t_{i+1}} - L^{t_i}, e_j^{(\lambda)} \rangle_U \, e_j^{(\lambda)} \bigg) = \sum_{i=1}^n X_i ( L^{t_{i+1}} - L^{t_i} ),
\end{align*}
completing the proof.
\end{proof}

Therefore, and since the space of simple processes is dense in the space of all predictable processes satisfying (\ref{int-cond-3}), see, e.g. \cite[Cor.~8.17]{P-Z-book}, the It\^{o} integral (\ref{def-integral-general}) coincides with that in \cite{P-Z-book} for every $L_2^0(H)$-valued, predictable process $X$ satisfying (\ref{int-cond-3}). In particular, for a driving Wiener process, it coincides with the It\^{o} integral from
\cite{Da_Prato, Prevot-Roeckner, Atma-book}.

By a standard localization argument, we can extend the definition of the It\^{o} integral to all predictable processes $X$ satisfying
\begin{align}\label{loc-integrable}
\mathbb{P} \bigg( \int_0^T \| X_s \|_{L_2^0(H)}^2 ds < \infty \bigg) = 1 \quad \text{for all $T > 0$.}
\end{align}
Since the respective spaces of predictable and adapted, measurable processes are isomorphic (see \cite{Ruediger-Tappe}), proceeding as in the \cite[Sec.~3.2]{Ruediger-Tappe}, we can further extend the definition of the It\^{o} integral to all adapted, measurable processes $X$ satisfying (\ref{loc-integrable}).

\end{document}